\documentclass[10pt]{amsart}
\usepackage{amsfonts}
\usepackage{mathrsfs}
\usepackage{amsmath}
\usepackage{graphicx}
\usepackage{caption}
\usepackage{leftidx}
\usepackage{amssymb}
\usepackage{amscd}
\usepackage{epsfig}
\usepackage{color}
\usepackage{xcolor}
\usepackage[percent]{overpic}
\usepackage[all]{xy}
\usepackage{color}
\usepackage{yfonts}
\usepackage{hyperref}
\usepackage[mathscr]{euscript}
\usepackage{tikz}

\newtheorem{thm}{Theorem}[section]
\newtheorem{prop}[thm]{Proposition}
\newtheorem{lemma}[thm]{Lemma}

\theoremstyle{definition}

\newtheorem{definition}[thm]{Definition}
\newtheorem{remark}[thm]{Remark}

\newcommand{\cB}{\mathcal{B}}
\newcommand{\cD}{\mathcal{D}}

\newcommand{\cF}{\mathcal{F}}

\newcommand{\cL}{\mathcal{L}}

\newcommand{\cN}{\mathcal{N}}
\newcommand{\cO}{\mathcal{O}}
\newcommand{\cS}{\mathcal{S}}

\newcommand{\cW}{\mathcal{W}}

\newcommand{\bC}{\mathbb{C}}

\newcommand{\bR}{\mathbb{R}}

\newcommand{\fg}{\mathfrak{g}}
\newcommand{\fh}{\mathfrak{h}}
\newcommand{\fn}{\mathfrak{n}}
\newcommand{\fb}{\mathfrak{b}}

\newcommand{\fS}{\mathfrak{S}}
\newcommand{\fU}{\mathfrak{U}}

\newcommand{\Ad}{\mathrm{Ad}}
\newcommand{\opp}{\mathrm{opp}}
\newcommand{\Sh}{\mathrm{Sh}}

\title[Big tilting sheaves as holomorphic Morse Branes]{Representing the Big tilting sheaves as holomorphic Morse Branes}
\author[Xin Jin]{Xin Jin}
\email{xin.jin@bc.edu}

\address{Department of Mathematics, Boston College}

\subjclass[2000]{}

\date{}


\numberwithin{equation}{section}

\begin{document}
\begin{abstract}
We introduce Morse branes in the Fukaya category of a holomorphic symplectic manifold, with the goal of constructing tilting objects in the category. We give a construction of a class of Morse branes in the cotangent bundles, and apply it to give the holomorphic branes that represent the big tilting sheaves on flag varieties. 
\end{abstract}

\maketitle
\tableofcontents

\section{Introduction}

For a complex semisimple Lie group $G$ and a Borel subgroup $B\subset G$ with its unipotent radical $N$, the category of $N$-equivariant perverse sheaves on $\cB=G/B$ corresponds to the principal block of the BGG Category $\cO$. The indecomposable tilting perverse sheaves form a natural basis for the category, and they are in bijection with the Schubert cells. One can also view the tilting sheaves from other perspectives, i.e. as $\cD$-modules via the Riemann-Hilbert correspondence or as Lagrangian branes in the Fukaya category $F(T^*\cB)$ via the Nadler-Zaslow correspondence. There have been several constructions of tilting objects as sheaves or $\cD$-modules, including certain averaging or limiting process (c.f. \cite{FrGa}, \cite{Nad06}, \cite{BeYu}, \cite{Cam}). In this paper, we construct the tilting object corresponding to the open Schubert cell, often referred as the \emph{big tilting}, as a holomorphic Lagrangian brane in the Fukaya category $F(T^*\cB)$. 

The construction is simple. Consider the moment map for the Hamiltonian $N$-action on $T^*\cB$, $\mu_N: T^*\cB\rightarrow \fn^*$, where $\fn$ is the Lie algebra of $N$. Take a non-degenerate character $\bar{e}$ of $\fn$ in $\fn^*$, then $L_{\bar{e}}=\mu_N^{-1}(\bar{e})$ is a closed (smooth) holomorphic Lagrangian in $T^*\cB$. It is just an $N$-orbit and we can equip it with a canonical brane structure to make it correspond to a perverse sheaf (c.f. \cite{Jin}).
\begin{thm}\label{big tilting thm}
The brane $L_{\bar{e}}$ corresponds to the big tilting sheaf on $\cB$, via the Nadler-Zaslow correspondence.
\end{thm}
The construction fits into a more general setting as \emph{Morse branes} in holomorphic symplectic manifolds that we will introduce below, and the consideration of Morse branes is largely motivated from the approach by Nadler \cite{Nad06} to construct tilting sheaves. We remark that a notable application of the holomorphic brane approach to tilting sheaves is that the branes come in a $\mathbb{C}^*$-family, and one can use it to give a geometric construction of the mixed Hodge structures on the tilting sheaves (in the sense of \cite{Saito}), and the construction is in the forthcoming work \cite{Jin4}. 

\subsection{Morse branes in holomorphic symplectic manifolds}
We will work in the setting that an exact holomorphic symplectic manifold $(M,\omega_\mathbb{C})$ is endowed with two commuting $\mathbb{C}^*$-actions: one is Hamiltonian and is denoted as $\mathbb{C}_X^*$, and the other, denoted as $\mathbb{C}_Z^*, $ scales $\omega_\mathbb{C}$ by a positive weight and it contracts $M$ to a compact core as $t\rightarrow 0$. We also assume that the $\mathbb{C}_X^*$-action has finitely many fixed points, and we will denote the union of their ascending (resp. descending) manifolds as $\Lambda_X$ (resp. $\Lambda_X^{opp}$). Both $\Lambda_X$ and $\Lambda_X^{opp}$ are holomorphic conical Lagrangians with respect to the $\mathbb{C}^*_Z$-action, by the commutativity condition of the two actions. We assume that $\Lambda_X$ and $\Lambda_X^{opp}$ are disjoint away from the compact core of $M$.

Consider the Fukaya category $F_{\Lambda_X}(M)$, whose objects are (closed) Lagrangian branes in $M$ that are dilated towards $\Lambda_X$ by $\mathbb{C}_Z^*$ as $t\rightarrow 0$. We call a brane $L\in F_{\Lambda_X}(M)$ a \emph{Morse brane} if it intersects $\Lambda_X^{opp}$ uniquely and transversely at a point in the smooth portion of $\Lambda_X^{opp}$. The name comes from the principle that it plays the role of calculating the ``microlocal stalk'' in $F_{\Lambda_X^{opp}}(M)$ at the intersection point (c.f. \cite{Nad14} and \cite{Jin}). 

We give a natural construction of a class of Morse branes in the situation when $M$ is the cotangent bundle of a complex projective variety with a contracting $\mathbb{C}_Z^*$-action on the fibers (of weight 1). The specialty of cotangent bundles is that 
if $k_0$ is the minimum of the positive weights of the $\mathbb{C}_X^*$-action on the tangent spaces of the fixed points, then we can use the flow of $\mathbb{C}_{X-k_0Z}^*$ to construct holomorphic Morse branes.  Here $\mathbb{C}_{X-k_0Z}^*$ is the subgroup in $\mathbb{C}_X^*\times \mathbb{C}_Z^*$ which is the graph of the group homomorphism $\mathbb{C}_X^*\rightarrow \mathbb{C}_Z^*, t\mapsto t^{-k_0}$.  We expect the construction to be generalized to some other holomorphic symplectic manifolds (e.g. hypertoric varieties, the resolution of the Slodowy slices) with more careful investigation of the weights of the two $\mathbb{C}^*$-actions, and we leave this for a future work. 

The construction goes as follows. Take a point $x$ in the fixed loci of $\mathbb{C}^*_{X-k_0Z}$, and take the ascending manifold of $x$ with respect to the $\mathbb{C}^*_{X-k_0 Z}$-action. By the weights condition, this is a (not necessarily closed) holomorphic Lagrangian submanifold and we denote it by $L_{x}$. The main theorem we get is the following.

\begin{thm}\label{L_x}
If $x\in(\Lambda_X^{opp})^{sm}$, then $L_x$ is a holomorphic Morse brane in $F_{\Lambda_X}(M)$.
\end{thm}

\subsection{Application to the construction of tilting objects}
In the case of a cotangent bundle, we have a $\mathbb{C}_X^*$-action on the base $K$ which induces the Hamiltonian $\mathbb{C}_X^*$-action on $T^*K$, and the Lagrangian $\Lambda_X$  (resp. $\Lambda_{X}^{opp}$) is the conormal variety to the stratification $\cS$ (resp. $\cS^{-}$) defined by the ascending (resp. descending) manifolds of the fixed points in $K$. 

In good situations, $\cS=\{S_\alpha\}$ and $\cS^{opp}=\{S_\alpha^{opp}\}$ are transverse to each other, and $\cS^{opp}$ is simple (see Definition \ref{simple}). Then Theorem \ref{big tilting thm} is a special case of a more general result. 

 \begin{thm}\label{L_x tilting}
 If $x\in(\Lambda_X^{opp})^{sm}$, then $L_x$ corresponds to a tilting sheaf on $K$ under the Nadler-Zaslow correspondence. 
 \end{thm} 

Once we have obtained Theorem \ref{L_x}, the proof of Theorem \ref{L_x tilting} follows from a similar argument as in \cite{Nad06}. Namely, the stalk (resp. costalk) of the corresponding sheaf on $S^{opp}_\alpha$ can be calculated by the microlocal stalk of the costandard (resp. standard) sheaf for $S^{opp}_\alpha$ at $x$, therefore they are concentrated in the right degrees.  

We expect Morse branes to give tilting objects in the Fukaya category of a wide class of holomorphic symplectic manifolds. In the case of symplectic resolutions, the Fukaya categories are expected to be equivalent to the category of modules over certain quantizations of the manifolds. Therefore the tilting branes are expected to correspond to tilting objects in certain representation categories. 

\subsection{Organization}
The paper is organized as follows. In Section \ref{section 2}, we recall some basic definitions and facts about constructible sheaves, perverse sheaves and tilting sheaves. In Section \ref{Fukaya_K}, we make the basic set-up for the Fukaya category of a holomorphic symplectic manifolds, and we also briefly review the definition of Fukaya categories and the Nadler-Zaslow correspondence. Next, we give the construction of a class of holomorphic Morse branes and the proof of Thm \ref{L_x} in Section \ref{section 4}. The proof is based on the analysis of the Morse-Bott flow of some combinations of the $\mathbb{C}_X^*$ and $\mathbb{C}_Z^*$-actions. Lastly, we give the big tilting brane in $T^*\cB$ and prove Theorem \ref{big tilting thm} in Section \ref{big tilting}. The exactly same proof applies to Theorem \ref{L_x tilting}.

\subsection{Acknowledgements}
The project of understanding tilting sheaves via holomorphic branes started from my graduate studies. I wanted to thank my Ph.D. advisor Prof. David Nadler for introducing me to this direction, and for his inspirations and help. I also wanted to thank Prof. Ivan Losev, David Treumann, Geordie Williamson,  Zhiwei Yun, Eric Zaslow and Dr. Justin Hilburn for useful conversations and feedbacks on this work. 

\section{Tilting perverse sheaves }\label{section 2}
\subsection{Constructible sheaves}
This subsection reviews some basic definitions and properties of constructible sheaves with the main purpose of introducing notations. We recommend \cite{Massey} for an introduction to the theory of constructible sheaves.
We will keep working in the subanalytic setting. 

Let $M$ be a real analytic manifold. Fix a Whitney stratification $\cS=\{S_\alpha\}$ on $M$. A sheaf $\cF$ of $\mathbb{C}$-vector spaces on $M$ is said to be \emph{constructible with respect to }$\cS$, if its pull-back to each stratum $i_{S_\alpha}^*\cF$ is locally constant. Let $D_{\cS}(M)$ (resp. $D(M)$) be the bounded derived category of complexes of sheaves whose cohomology sheaves are all constructible with respect to $\cS$ (resp. with respect to some stratification). Let $Sh_{\cS}(M)$ (resp. $Sh(M)$) be the natural dg-enhancement of $D_{\cS}(M)$ (resp. $D(M)$). We will always refer to an object in $Sh(M)$ a sheaf rather than a complex of sheaves. 

For any map $f: M_1\rightarrow M_2$ between two analytic manifolds, there are standard operations $f_*, f_!: Sh(M_1)\rightarrow Sh(M_2)$, $f^*, f^!: Sh(M_2)\rightarrow Sh(M_1)$, where all of our functors have been derived and we always omit the derived notation. There is also the Verdier duality $\mathbb{D}: Sh(M)\overset{\sim}{\rightarrow} Sh(M)^{op}$, which intertwines the $*,!$ functors, i.e. $f_!=\mathbb{D}f_*\mathbb{D}$ and $f^!=\mathbb{D}f^*\mathbb{D}$. 

For any open embedding $i: U\hookrightarrow M$ and closed embedding of the complement $j: Z\hookrightarrow M$, there are the standard triangles
$$i_!i^!\cF\rightarrow\cF\rightarrow j_*j^*\cF,\ j_!j^!\cF\rightarrow\cF\rightarrow i_*i^*\cF,$$
from which it is not hard to deduce that $Sh_\cS(M)$ is generated by $i_{S_\alpha*}\cL_{S_\alpha}, S_\alpha\in \cS$, where $\cL_{S_\alpha}$ ranges in the set of irreducible local systems on $S_\alpha$.

\subsection{Perverse sheaves and tilting sheaves}
Here we recall the basic definitions and properties of perverse sheaves and tilting sheaves. We refer the reader to \cite{Goresky}, \cite{Kashiwara} for more discussions on perverse sheaves and \cite{BBM} on tilting sheaves.

\subsubsection{Perverse sheaves}
The most natural definition of perverse sheaves may be through the Riemann-Hilbert correspondence. For a complex analytic manifold $M$, the Riemann-Hilbert correspondence gives an equivalence between the bounded derived category of regular holonomic $\cD$-modules and $D(M)$. The obvious $t$-structure on the $\cD$-module side induces an interesting $t$-structure on $D(M)$, which is called the \emph{perverse} $t$-structure. The \emph{perverse sheaves} are the objects in the heart of the $t$-structure. In other words,  a perverse sheaf corresponds to a single regular holonomic $\cD$-module. 

There are other characterizations of perverse sheaves. A commonly used one is the following definition through the degrees of cohomological (co)stalks of sheaves. Let $\cF$ be a sheaf that is constructible with respect to a complex stratification $\cS=\{S_\alpha\}$.

 \begin{definition}
A sheaf $\cF$ is \emph{perverse} if the followings hold for all $S_\alpha\in\cS$:\\
(1) $H^\bullet(i_{S_\alpha}^*\cF)=0$ for all $\bullet>-\dim_\mathbb{C} S_\alpha$;\\
(2) $H^\bullet(i_{S_\alpha}^!\cF)=0$ for all $\bullet<-\dim_\mathbb{C} S_\alpha$.
  \end{definition} 
  
There is another natural characterization of perverse sheaves through microlocal stalks (also called local Morse groups or vanishing cycles). Let's first briefly review the definition of microlocal stalks. Microlocal stalks are well defined in the real setting (c.f. \cite{Goresky}), however, we will restrict ourselves to the complex setting for simplicity. For any covector $(x,\xi)\in\Lambda_{\cS}^{sm}$, we choose a generic germ of holomorphic function $F$ near $x$ such that $F(x)=0$ and $dF_x=\xi$. Here the genericity condition can be interpreted as that the graph of $dF$ as a germ of Lagrangian in $T^*M$ is transverse to $\Lambda_\cS$ at $(x,\xi)$. 

\begin{definition}
The \emph{microlocal stalk} of $\cF\in Sh_{\cS}(M)$ at $(x,\xi)$, denoted as $M_{x,\xi}(\cF)$ is defined to be 
$$M_{x,\xi}(\cF)=\Gamma(B_\epsilon(x), B_\epsilon(x)\cap \{\text{Re}F<0\}; \cF),$$
for $\epsilon>0$ sufficiently small. 
\end{definition}

Now we can define the \emph{singular support} of a sheaf $\cF\in Sh_\cS(M)$ to be 
$$SS(\cF)=\overline{\{(x,\xi)\in\Lambda_{\cS}^{sm}: M_{x,F}(\cF)\not\simeq 0\}}.$$

One important feature about microlocal stalk is that it is perverse $t$-exact. Moreover, we have the following microlocal characterization of perverse sheaves.
\begin{prop}
A sheaf $\cF$ is perverse if and only if all of its microlocal stalks are concentrated in degree 0. 
\end{prop}

\subsubsection{Tilting sheaves}
Tilting sheaves form a special kind of perverse sheaves. Under some natural assumptions on the  stratification $\cS$, the indecomposable tilting sheaves form a natural basis for the category of perverse sheaves. 

\begin{definition}\label{simple}
A complex stratification $\cS=\{S_\alpha\}$ is called \emph{simple} if the frontier of each stratum $\overline{S_\alpha}-S_\alpha$ is a Cartier divisor in $\overline{S_\alpha}$. 
\end{definition}
It is a standard fact that the Schubert stratification on a flag variety $\cB=G/B$ is simple. This is because each stratum is isomorphic to a unipotent subgroup of $G$, so it is affine, hence the inclusion from each stratum to $G/B$ is affine. If $\cS$ is simple, then the standard and costandard sheaves $i_*\cL_{S_\alpha}[-\dim S_\alpha]$, \\
$i_!\cL_{S_\alpha}[-\dim S_\alpha]$ are both perverse sheaves, for any local system $\cL_{S_\alpha}$ on $S_\alpha$. 

\begin{definition}
A sheaf $\cF\in Sh_{\cS}(M)$ is \emph{tilting} if for all $S_\alpha\in\cS$, we have\\
(1) $H^\bullet(i_{S_\alpha}^*\cF)=0$ for all $\bullet\neq -\dim_\mathbb{C} S_\alpha$;\\
(2) $H^\bullet(i_{S_\alpha}^!\cF)=0$ for all $\bullet\neq-\dim_\mathbb{C} S_\alpha$.
\end{definition}

\begin{prop}
If $\cS$ is simple and $\pi_1(S_\alpha)=\pi_2(S_\alpha)=0$ for every $S_\alpha\in\cS$, then there is a unique indecomposable tilting perverse sheaf supported on each $\overline{S_\alpha}$, and  this gives a bijection between indecomposable tilting perverse sheaves and the strata in $\cS$. 
\end{prop}

\section{Fukaya categories on holomorphic symplectic varieties}\label{Fukaya_K}

Let $M$ be a (quasi-projective) holomorphic symplectic variety with an exact holomorphic symplectic form $\omega_{\mathbb{C}}$. 

\subsection{Two $\mathbb{C}^*$-actions}\label{C*-action}
We assume that $M$ is equipped with two commuting (algebraic) $\mathbb{C}^*$-actions: $\mathbb{C}^*_X$ and $\mathbb{C}^*_Z$, where $X$ and $Z$ denote for the integral vector fields of the corresponding $U(1)$-actions respectively. Similarly for any integral combination of the vector fields of $X$ and $Z$ we can define the corresponding $\mathbb{C}^*$-action. 

The $\mathbb{C}^*_X$-action should be Hamiltonian with respect to $\omega_{\mathbb{C}}$, and it should have finitely many fixed points. We index the fixed points by $x_\alpha,\alpha\in I$, and use $\fS_X(x_\alpha)$ (resp. $\fU_X(x_\alpha)$) to denote the ascending manifold (resp. descending manifold) of $x_\alpha$.  
There is a natural partial ordering on the fixed point set $I$, namely $x_{\alpha}\prec x_{\beta}$ if $x_\alpha\in\overline{\fS}_X(x_\beta)$. 
The ascending manifold of each fixed point is a holomorphic Lagrangian manifold in $M$, and we will denote the union of them by $\Lambda_X$. 

The $\mathbb{C}_Z^*$-action contracts $M$ to a compact core, denoted as $\text{Core}(M)$, and it acts on $\omega_\mathbb{C}$ by weight $k$, for some integer $k\geq 1$. By the commutativity assumption, $\Lambda_X$ is conical with respect to the $\mathbb{C}_Z^*$-action.

\subsection{Examples}
A class of interesting examples of holomorphic symplectic manifolds are the conical symplectic resolutions. We refer the readers to the definition and a list of examples in Section 2 of \cite{BPW}. 

In this paper we will mostly focus on the case when $M=T^*K$ is the cotangent bundle of a complex projective variety $K$, the $\mathbb{C}_X^*$-action will be the induced Hamiltonian action from a given $\mathbb{C}_X^*$-action on $K$ (with isolated fixed points), and the $\mathbb{C}_Z^*$-action will be the contraction on the cotangent fibers. In particular, we have $k=1$.

\subsection{The Fukaya category $F_{\Lambda_X}(M)$}

\subsubsection{A brief review of the Fukaya category in the real setting}
For any real exact symplectic manifold $(M,\omega)$ with a conical end with respect to the Liouville flow for a preferred primitive of $\omega$ (such a manifold is called a Liouville manifold), one can define its infinitesimal Fukaya category\footnote{We always assume the Fukaya category to be triangulated. }, denoted by $F(M)$. The definition is originated from \cite{NZ} in the cotangent bundle case and can be generalized to Liouville manifolds. The book \cite{Seidel} treats the case of Lefschetz fibrations.  For an expanded review of infinitesimal Fukaya categories, see \cite[Appendix C]{Jin}. Roughly speaking, an object in the Fukaya category is a (complex of) Lagrangian brane(s) $(L,\Phi, P)$ consisting of the data\footnote{The brane structure also includes a local system (equivalently, a vector bundle with a flat connection) on $L$. For simplicity, in this paper, we will assume that the local system is always trivial of rank 1.} of a properly embedded Lagrangian submanifold $L$, a grading $\Phi: L\rightarrow\mathbb{R}$, and a relative Pin-structure on $L$. In the following, to make the notations simple, we usually denote a brane only by its underlying Lagrangian submanifold when there is no cause of confusion. 
Moreover, one compactifies $M$ by the conical structure on the ends to $\overline{M}=M\cup M^\infty$, where $M^\infty$ is the contact boundary of $M$ which is also referred as the infinity of $M$. We also require that $L$ is well-behaved near the infinity of $M$ in the sense that $L^\infty=\overline{L}\cap M^\infty$ is a Legendrian subset of $M^\infty$, which can be equivalently described as $\lim\limits_{t\rightarrow 0^+}t\cdot L$ is contained in a conical Lagrangian. 

The morphism between two objects $(L_1,\Phi_1, P_1)$ and $(L_2, \Phi_2, P_2)$ is the Floer complex $CF(L_1,L_2)=(\bigoplus\limits_{p\in L_1\cap L_2}\mathbb{C}\cdot p [-\mathrm{deg }\
 p], \mu_1)$, where $\mu_1$ is defined by counting pseudo-holomorphic discs bounded by the two Lagrangians. The degree of $p$, denoted as $\mathrm{deg}\ p$, depends on the gradings $\Phi_1$ and $\Phi_2$. The relative Pin-structures also enter into the story because these are needed to give an orientation of the (0-dimensional) moduli spaces of pseudo-holomorphic strips, so that one can count the points. Of course, implicit in the definition is the transversality between $L_1$ and $L_2$ and certain standard treatment of $L_1^\infty$ and $L_2^\infty$ if they overlap. 
 
The composition of morphisms 
$$\mu_2:CF(L_2, L_3)\otimes CF(L_1,L_2)\rightarrow CF(L_1, L_3)$$
is defined by counting pseudo-holomorphic triangles bounded by the three Lagrangians. There are also higher compositions $\mu_n, n\geq 3$ which are defined by counting pseudo-holomorphic polygons. The sequence $\{\mu_n\}_{n\geq 1}$ satisfies the $A_\infty$-relation, which makes the Fukaya category into an $A_\infty$-category. 

Since we will only use a short list of theorems or facts about the  Fukaya categories, we find it not necessary to go through the long story of the subject. We will review the statements we need in the next subsection and  refer the reader to \cite{Seidel}, \cite{Auroux}  and \cite{NZ} for more details on the definition of Fukaya categories. 

\subsubsection{The subcategory $F_\Lambda(M)$}
Continuing on the real setting, for any conical Lagrangian $\Lambda\subset M$, we define the full subcategory $F_\Lambda(M)^{\mathrm{naive}}$ to be generated by objects $L$ with $L^\infty\subset \Lambda^\infty$. We put the superscript ``naive'' because the actual definition of  $F_\Lambda(M)$ is defined microlocally, which corresponds to $Sh_\Lambda(K)$ when $M=T^*K$. Given an $L\in F(M)$, for any $\xi\in (L^\infty)^{sm}$, one can construct a Lagrangian disc $L_{\xi}$ (which is also an object in $F(M)$) whose infinity is disjoint from $L^\infty$ and which intersects the cone over $L^\infty$ transversely at a unique point in the ray pointing to $\xi$.  For more details of the construction of $L_{\xi}$, we refer the reader to Section 3.7 in \cite{Nad14} and Section 4 in \cite{Jin} (in the cotangent bundle case). Once such a brane $L_{\xi}$ can be constructed for every $\xi$, we can define the \emph{microlocal support} of a brane $L$, which is a conical Lagrangian. Then $F_\Lambda(M)$ is the full subcategory generated by branes whose microlocal support is contained in $\Lambda^\infty$. 

In this paper, we will be mostly interested in the objects in $F_\Lambda(M)^{\mathrm{naive}}$, so it is not harmful to keep that as an intuitive replacement of $F_\Lambda(M)$.

\subsubsection{$F(M,\omega_\mathbb{C})$ and $F_{\Lambda_X}(M)$}

In the holomorphic symplectic setting, as we started with, we take the real part of $\omega_\mathbb{C}$ and the $\mathbb{R}_+$-factor in $\mathbb{C}_Z^*$ to serve as the Liouville flow, then these fit into the real setting, and give us the  Fukaya category $F(M,\omega_\mathbb{C})$. Similarly, we can define $F_{\Lambda_X}(M)$ to be the subcategory of $F(M,\omega_\mathbb{C})$ in the real setting.

There are some special features about the Fukaya category of a holomorphic symplectic manifold. For example, one can do a projective compactification $\overline{M}_\mathbb{C}=M\cup M^\infty_\mathbb{C}$ of $M$ using the $\mathbb{C}^*_Z$-action, so that $M^\infty=(M-\text{Core}(M))/\mathbb{C}_Z^*$ (we will omit the subscript $\mathbb{C}$ from now on) is a divisor in $\overline{M}$. Moreover, there is a specific class of Lagrangians--the holomorphic Lagrangians. In \cite{Jin}, it is proved that any holomorphic Lagrangian brane in $M=T^*K$ represents a perverse sheaf on $K$, under the Nadler-Zaslow correspondence. Hence one could roughly think of the class of the holomorphic branes as the heart of a $t$-structure on the Fukaya category\footnote{This is not a precise statement, since not every perverse sheaf can be represented by a holomoprhic brane.}.

\subsection{The Nadler-Zaslow correspondence}\label{NZC}
Given a compact real analytic manifold $K$, the Nadler-Zaslow correspondence gives a quasi-equivalence between the Fukaya category $F(T^*K)$ and the dg-category $Sh(K) $ of constructible sheaves on $K$. The theorem also holds for a given microlocal support condition, i.e. given a conical Lagrangian $\Lambda\subset T^*K$ (containing the zero-section), we have $F_\Lambda(T^*K)\simeq Sh_{\Lambda}(K)$, where $Sh_\Lambda(K)$ denotes for the full subcategory consisting of sheaves whose singular support is contained in $\Lambda$.

We will collect some of the results involved in the Nadler-Zaslow correspondence that we will use in later sections without proof. We refer the interested reader to \cite{NZ} and \cite{Nad09} for more details. In the following, we will fix a Whitney stratification $\cS=\{S_\alpha\}$ on $K$ such that each stratum is connected and is a cell, and we will always work in the subanalytic setting. 

\begin{itemize}
\item \emph{(Co)Standard branes.}

For each stratum $S_\alpha\in\cS$, one can define a \emph{standard brane} on it, denoted as $L_{S_\alpha}$ as follows. Pick a function $m_\alpha: K\rightarrow \mathbb{R}$ such that $m_\alpha>0$ on $S_\alpha$ and $m_\alpha=0$ on $K-S_\alpha$. Now define $L_{S_\alpha}$ to be $\Gamma_{d\log m_\alpha}+T_{S_\alpha}^*K$. It is shown in \cite{NZ} that $L_{S_\alpha}$ can be equipped with a canonical grading and a canonical Pin-structure, so we will refer $L_{S_\alpha}$ as the standard brane on $S_\alpha$. Note that 
$L_{S_\alpha}$ as an object in $F(T^*K)$ doesn't depend on the choices of $m_\alpha$. 
The involution on $T^*K$ that negates the cotangent vectors correspond to the Verdier duality on $Sh(K)$.We will call the involution of $L_{S_\alpha}$ a \emph{costandard brane}.

\item \emph{Generators of $F_{\Lambda_{\cS}}(T^*K)$.}

Under the Nadler-Zaslow correspondence, each standard brane $L_{S_\alpha}$ goes to the standard sheaf $i_{S_\alpha*}\mathbb{C}_{S_\alpha}$, and the involution of $L_{S_\alpha}$ goes to the costandard sheaf $i_{S_\alpha!}\mathbb{C}_{S_\alpha}$, where $i_{S_\alpha}: S_\alpha\hookrightarrow K$ is the embedding. If we put a standard or costandard sheaf (resp. brane) for each stratum, then they will generate $Sh_{\Lambda_\cS}(K)$ (resp. $F_{\Lambda_{\cS}}(T^*K)$) by taking shifts and iterated cones.

\end{itemize}

\section{Holomorphic Morse branes in $F_{\Lambda_X}(M)$}\label{section 4}
We will continue on the set-up for the Fukaya category of a holomorphic symplectic manifold in Section \ref{Fukaya_K}. 

\subsection{Definition of Morse branes in $F_{\Lambda_X}(M)$}

Let $\Lambda_X^{\opp}$ be the union of the descending manifolds of $\mathbb{C}_{X}^*$. We assume that $\Lambda_X$ and $\Lambda_X^{\opp}$ are disjoint away from the compact core of $M$. 

\begin{definition}\label{def: Morse brane}
A Lagrangian brane $L$ in $F_{\Lambda_X}(M)$ is called a \emph{Morse brane}, if it satisfies that $L$ intersects $\Lambda_X^{\opp}$ in a single point that is contained in the smooth part of $\Lambda_X^{\opp}$, and the intersection is transverse. 
\end{definition}
The consideration of Morse branes is largely motivated by the results in \cite{Nad06}, in which the author 
constructed tilting perverse sheaves on the flag variety $\cB$ by means of Morse theory. We will see the applications of the notion of Morse branes in the construction of big tilting sheaves in Section \ref{big tilting}. We also remark that there is an intimate relation between the Morse brane here and the so called local Morse branes in \cite{Jin}. In \cite{Jin}, local Morse branes are introduced to represent the Morse kernel (vanishing cycle functor) in the Fukaya category at a given smooth point of a holomorphic conical Lagrangian in the cotangent bundle of a complex manifold.  One can generalize the construction to a holomorphic symplectic manifold $M$ (with conical ends) since the construction is completely local. In our current situation, a (holomorphic) Morse brane is definitely a local Morse brane, but it is more rigid and relies on the global geometry of $M$, for it satisfies additional microlocal condition from $\Lambda_X$. 

\subsection{Construction of a class of holomorphic Morse branes in cotangent bundles}\label{Morse brane}

In this section, we assume that $M$ is the cotangent bundle of a smooth projective variety. The action by $\mathbb{C}_Z^*$ is dilating the fibers with weight 1, and we assume that 
\begin{align}\label{weight assumption}
&\text{the minimum of the positive weights of } \mathbb{C}^*_X\text{ on the tangent spaces} \\
\nonumber&\text{at the fixed points is } k_0. 
\end{align}
We will use $a_X,\ a_Z,\ a_{X-k_0Z}: \mathbb{C}^*\rightarrow \text{Aut}(M)$ to denote the action of $\mathbb{C}^*_X$, $\mathbb{C}^*_Z$ and $\mathbb{C}^*_{X-k_0Z}$ on $M$, respectively. Again, we index the fixed points of $\mathbb{C}_X^*$ by $x_\alpha, \alpha\in I$. We will  denote each fixed locus of $\mathbb{C}^*_{X-k_0Z}$ containing a $\mathbb{C}_X^*$-fixed point $x_\alpha$ by $E_\alpha$.
\begin{lemma}\label{ascending Lag}
For any $x$ in the fixed loci of $\mathbb{C}_{X-k_0Z}^*$, the ascending manifold $\mathfrak{S}_{X-k_0Z}(x)$ is a holomorphic Lagrangian submanifold (not necessarily closed). 
\end{lemma}
\begin{proof}
First, $x$ must lie in the descending manifold of the $\mathbb{C}^*_X$-fixed point $$y=\lim\limits_{t\rightarrow 0}a_Z(t)\cdot x=\lim\limits_{t\rightarrow 0}a_X(t)\cdot x,$$ 
therefore it belongs to $E_\alpha$ for some $\alpha\in I$. 

Since the action of $\mathbb{C}_{X-k_0Z}^*$ is Morse-Bott, the decomposition of the tangent space at $x$ into weight spaces  is the same as that at $x_\alpha$. By the assumption (\ref{weight assumption}), we know that  the ascending manifold of $x_\alpha$ with respect to $\mathbb{C}_{X-k_0Z}^*$ is the same as the ascending manifold with respect to $\mathbb{C}_X^*$, thus the negative weight space of $\mathbb{C}_{X-k_0Z}^*$ has the dimension of a Lagrangian. Now at $x$, we only need to show that in a small neighborhood, the ascending manifold $\fS_{X-k_0Z}(x)$ is isotropic, since $\mathbb{C}^*_{X-k_0Z}$ scales $\omega_\mathbb{C}$ with weight $-k=-k_0$. First, the tangent space at $x$ is isotropic by a similar reason of weights: the negative weights for $X-Z$ are at most $-2k_0$. To show that near $x$ we have $\mathfrak{S}_{X-k_0Z}(x)$ locally be a Lagrangian, we identify a neighborhood of $0$ in $T_xM$ with a neighborhood of $x$ in $M$ by an $a_{X-k_0Z}(\mathbb{R})$-equivariant diffeomorphism, and use the equivariant version of Moser's argument to modify the diffeomorphism into a local equivariant symplectomorphism. 

\end{proof}

\begin{remark}
It is easy to see that $C(L):=\lim\limits_{t\rightarrow 0}a_Z(t)\cdot L$ is both $\mathbb{C}^*_Z$ and $\mathbb{C}^*_X$-invariant. However, we cannot conclude that $C(L)$ is contained in the conical Lagrangian $\Lambda_X$. 
\end{remark}

We will denote every Lagrangian constructed in Lemma \ref{ascending Lag} by $L_{\alpha,x}$, for  $x\in E_\alpha$. Now we work with the projective compactification of $M$ with respect to the action of $\mathbb{C}_Z^*$, defined by 
$$\overline{M}=(M\times \mathbb{C}-\mathrm{Core}(M)\times \{0\})/\mathbb{C}_Z^*.$$
Since our $M$ is the cotangent bundle of a projective variety, $\overline{M}$ is again projective. 
The action of $\mathbb{C}_X^*$ and $\mathbb{C}_Z^*$ both extend to $\overline{M}$ by keeping their actions on $M$ and acting trivially on the extra factor $\mathbb{C}$. In particular, they will preserve $M^\infty=\overline{M}-M$. We will denote the projectivization of a conical line $\mathbb{C}_Z^*\cdot v$ in $M$ by $[v]\in M^\infty$.

Now by basic properties of algebraic $\mathbb{C}^*$-actions on smooth projective varieties and its relations to Morse theory (c.f. \cite{Ginzburg} Section 2.4), we can deduce the following.
\begin{thm}\label{tilting}
If $E_{\alpha}\not\subset\bigcup\limits_{x_\beta\prec x_\alpha}\overline{\mathfrak{U}}_{X}(x_\beta)$, then for any $x\in E_{\alpha}-\bigcup\limits_{x_\beta\prec x_\alpha}\overline{\mathfrak{U}}_{X}(x_\beta)$, $L_{\alpha,x}$ is a Morse brane in $F_{\Lambda_X}(M)$ with $\overline{L}_{\alpha,x}\cap \overline{\Lambda_X^{\opp}}=\{x\}$. 
\end{thm} 
\begin{proof}
First, we have $\overline{\mathfrak{U}}_{X}(x_{\beta})= \overline{\mathfrak{U}}_{X-k_0Z}(E_{\beta})$ by Assumption (\ref{weight assumption}). Next, we claim that the boundary of $L_{\alpha, x}$ consists of points in $\overline{M}$ that can be connected to $x$ by piecewise flow lines, which are usually called broken flow lines. This follows from the properties of \emph{finite volume flow} in \cite{Harvey}, and can be argued in the same way as Lemma 3.4 in \emph{loc. cit.} More explicitly, one can construct a Kahler metric on $\overline{M}$ and a Morse-Bott function whose gradient flow gives the $\bR_+\subset \bC_{X-k_0Z}^*$ action (c.f. \cite[Section 2.4]{Ginzburg} or \cite[Section 9]{Harvey}). Then for any sequence $y_i\in \mathfrak{S}_x$, because the flow lines from $y_i$ to $x$ are of bounded lengths, up to passing to a subsequence, the flow lines converge to a broken flow line, i.e. there is $y_\infty=\lim\limits_{i\rightarrow\infty}y_i$ and a finite  sequence of critical points $p_1,\cdots, p_k$ such that the flow connects $y_\infty$ to $p_1$, then $p_i$ to $p_{i+1}, 1\leq i\leq k-1$ and lastly $p_{k}$ to $x$.

Now by the assumption that $E_{\alpha}\not\subset\bigcup\limits_{x_\beta\prec x_\alpha}\overline{\mathfrak{U}}_{X}(x_\beta)$, there is no flow line of $X-k_0Z$ that travels from $E_\beta$ to $x$. Therefore, for any broken flow line ending on $x$, the last portion must start from a point $[v]$ on a critical manifold inside $M^\infty$. We claim that $[v]$ is lying in $\Lambda_X^\infty$. Note that the critical manifolds in  $M^\infty$ are exactly the projectivization of the conical lines in $M$ that are fixed (pointwise) by $\mathbb{C}^*_{X+kZ}$ for some nonzero integer $k$. In particular, this says that $[v]\in \Lambda_X^\infty$ if and only if the conical line corresponding to $[v]$ is fixed by  $\mathbb{C}^*_{X+k_{[v]}Z}$ for some positive integer $k_{[v]}$. Suppose the contrary, we have $k_{[v]}<0$, this would imply the descending manifold of $[v]$ under the flow of $\mathbb{C}_{X-k_0Z}^*$ is contained in $M^\infty$, which cannot be true, so the claim follows.  Since $\Lambda_X^\infty\cap(\Lambda_X^{opp})^\infty=\emptyset$ by assumption, we can conclude that the broken line is contained in $M^\infty$ except for the last portion. 

 Now we can model the piece of flow line in $M^\infty$ ending at $[v]$ by a flow of $\mathbb{C}_{X+k_{[v]}Z}^*$ ending at a point $v_0\in \mathbb{C}_Z^*\cdot v$ in $M$, which means that the projectivization of the latter flow line in $M^\infty$ will be equal to that piece of flow line (here we have used again that $\mathbb{C}_{X+k_{[v]}Z}^*$ gives rise to a Morse-Bott flow on $\overline{M}$). If the starting point of the flow line modeled on is away from the zero section, then by rescaling it with $a_{-k_{[v]}Z}(t)$ with respect to some parametrization (so that we get a flow line of $\mathbb{C}_X^*$), it is clear that the whole flow line at infinity lies in $\Lambda_X^\infty$.
On the other hand, if the flow line  starts at some fixed point $x_\beta$ of $\mathbb{C}_X^*$, then there are two cases after rescaling the flow line of $\mathbb{C}_{X+k_{[v]}Z}^*$ in $M$ by $a_{-k_{[v]}Z}(t)$: 
one is at $0$ the flow line approaches something away from the compact core, the other is at $0$ it remains to be at $x_\beta$. The first case directly implies that the flow line in $M^\infty$ is contained in $\Lambda_X^\infty$, and the second implies that $\Lambda_X^\infty\cap (\Lambda_X^{opp})^\infty\neq \emptyset$, which is a contradiction.
By induction on the pieces of the broken flow line (from $\infty$ to $0$), we get that the whole broken line is lying in $\Lambda_X^\infty$ except for the first piece. This completes the proof  that $L_{\alpha, x}$ satisfies the geometric conditions in  Definition \ref{def: Morse brane}.

Lastly, we show the compactness of the moduli space of $J$-holomorphic discs bounding $L_{\alpha, x}$ and a finite collection of branes $L_1,\cdots, L_k$, which proves that $L_{\alpha, x}$ is a well defined object in $F_{\Lambda_X}(M)$. Here $J$ is a $\Re\omega_{\mathbb{C}}$-compatible almost complex structure that is also compatible with the conical structure of $M$ (c.f. \cite[(7b)]{Seidel}).  We remark that in \cite{NZ}, certain tameness conditions (and more generally a family of tame perturbations) are imposed on the branes in the Fukaya category, but this is purely for the sake of ensuring compactness of the moduli space of $J$-holomorphic discs.

Let $\varphi_1^t$ (resp. $\bR^+_{X-k_0Z}$) denote for the radial flow (resp. radial action) associated with $\bC^*_{X-k_0Z}$. Given $L_0=L_{\alpha, x}$ and (tame) branes $L_1,\cdots, L_k$, which satisfy the generic condition that they don't intersect at infinity and their intersections are transverse (here the order doesn't really matter), since $X$ is Hamiltonian, the calculation of Floer cochains and the $A_\infty$-maps are invariant under the flow of $\varphi_1^t$. More precisely, one needs to first modify $X$ so that it becomes $0$ outside a neighborhood of $L_{\alpha,x}$ (but it remains the same on a smaller neighborhood of $L_{\alpha,x}$), and in particular, it should be $0$ in a neighborhood of $\bigcup\limits_{i=1}^kL_i^\infty$. Let $H$ be a contact hypersurface in $M$ such that $M\cong M_0\underset{H}{\cup} ([0,\infty)\times H)$, where $\partial M_0=H$ and the Liouville vector field on $M$ is corresponding to the vector field $\partial_r$ on the factor $[0,\infty)$ with coordinate $r$. Choose $K\gg 0$ and let $M_{\leq K}=M-((K,\infty)\times H)$ and $M_{>K}=((K,\infty)\times H$. 

Now we show that after replacing $L_i$ by $\varphi_1^t(L_i), i=1,\cdots, k$ for $t$ sufficiently large (note that $L_0=L_{\alpha, x}$ is invariant under $\varphi_1^t$), any $J$-holomorphic polygon $u: (S,\partial S)\rightarrow (M, L_0\cup\bigcup\limits_{i=1}^kL_i)$ satisfies that $u(\partial S)\subset M_{\leq K}$ for a fixed sufficiently large $K$. Here $\partial S=\bigcup\limits_{i=0}^kC_i$ and $u(C_i)\subset L_i$. 
Note that $L_{x,\alpha}\cap M_{\leq K}$ is tame in the sense of \cite[Definition 4.7.1]{Sikorav}, so we can apply \cite[Proposition 4.7.2(ii)]{Sikorav} using appropriate $r_{L_0}, C_4(L_0)$ on $L_{x,\alpha}\cap M_{\leq K-1}$. Now fix a small neighborhood $U_x$ of $x$ in $L_{\alpha, x}$. It is clear that if a curve on $L_{\alpha, x}$ has one end in $U_x$ and the other outside $L_{x,\alpha}\cap M_{\leq K}$, then there needs at least $N_K$-balls of radius $r_{L_0}$ to cover it, for a fixed $N_K\gg 0$. Since $X$ is 0 in a neighborhood of $\bigcup\limits_{i=1}^kL_i^\infty$, by enlarging $K$ and replacing $r_{L_0}$ by a smaller one if needed (these choices can be made once for all), we can be sure that at least one of the balls does not intersect $(\bigcup\limits_{i=1}^k\bigcup\limits_{t>T}\varphi_1^t(L_i))\cap M_{>K-2}$ for some (fixed) $T\gg 0$. It follows then from  \cite[Proposition 4.7.2(ii)]{Sikorav}, the area of a $J$-holomorphic disc $u$ as above satisfying $u(C_0)\cap U_x\neq \emptyset$ and $u(C_0)\cap M_{>K}\neq \emptyset$ has a uniform lower bound $\epsilon>0$, which does not depend on $t$ for $t>T$.

Since $L_{x,\alpha}$ is the ascending manifold of $x$ with respect to $\varphi_1^t$, again after replacing $L_i$ by $\varphi_1^t(L_i), i=1,\cdots, k$ for $t$ sufficiently large, the intersection points $L_{\alpha, x}\cap L_1$ and $L_k\cap L_{\alpha,x}$ are getting inside $U_x$. Since $u(C_0)$ has boundary points contained in $L_{\alpha,x}\cap (L_1\cup L_k)$,  if $u(C_0)\cap (L_{x,\alpha}\cap M_{>K})\neq \emptyset$, then by the conclusion above, the area of $u$ is at least $\epsilon$. However, since $\bR^+_{X-k_0Z}$ scales the area of a $J$-holomorphic disc by weight $-k_0$ and the area of any disc only depends on the intersection points of the Lagrangians that they connect (this is a standard fact for exact Lagrangians), we can
always make $\text{area}(u)<\epsilon$ for every $u$ after a sufficient dilation by $\bR^+_{X-k_0Z}$, so then $u(C_0)\subset L_{x,\alpha}\cap M_{\leq K}$. By the tameness condition imposed on $L_1,\cdots, L_k$, we have $u(C_i)\subset M_{\leq K}, i=1,\cdots,k$ for a fixed large $K$ as well. So we finish the proof that $u(\partial S)\subset M_{\leq K}$. Lastly, by the maximum principle (c.f. \cite[Lemma 7.4]{Seidel}), we also have $u(S)\subset M_{\leq K}$. So the compactness of the moduli space is established as desired.
\end{proof}

\section{The big tilting branes in $T^*\cB$}\label{big tilting}
Let $G$ be a semisimple Lie group over $\bC$, $B\subset G$ be a Borel subgroup and $\cB$ be the flag variety $G/B$.  
Fix a maximal torus $H\subset B$. Let $B^-$ be the opposite Borel subgroup, and $N\subset B$ (resp. $N^-\subset B^-$) be the unipotent radical of $B$ (resp. $B^-$).  Let $\fg, \fb, \fb^-,\fn, \fn^-, \fh$ be the Lie algebra of $G, B, B^-, N, N^-, H$ respectively. 
For a general Borel $\fb_x$, we will use $\fn_{\fb_x}$ to denote its nilradical. 
Let $\Delta, \Phi^+$ and $\Phi^-$ denote respectively the set of simple, positive and negative roots. Let $W=N_G(H)/H$ be the Weyl group of $G$.  Let $\mathcal{S}=\{S_w\}_{w\in W}$ (resp. $\mathcal{S}^-=\{S_w^-\}_{w\in W}$) be the Schubert stratification (resp. opposite Schubert stratification) on $\cB$ determined by the orbits of $N$ (resp. $N^-$).  Fixing  the coweight in $\fh$ whose pairing with the simple roots are all $-1$, usually called $\check{\rho}$, its induced $\bC^*$-action on $\cB$ has fixed points naturally indexed by $W$, denoted as $p_w, w\in W$, and the ascending (resp. descending) manifolds of each of the fix points $p_w$ coincide with $S_w$ (resp. $S_w^-$).  Let $s_w: S_w\hookrightarrow \cB$ and $s_w^-: S_w^-\hookrightarrow \cB$ (resp. $i_{p_w}: p_w\hookrightarrow \cB$) be the embeddings of the strata (resp. fixed points).
 
The $\bC^*$-action on $\cB$ induces a Hamiltonian action on $T^*\cB$,  which we will use as the $\bC_X^*$-action as in Section \ref{C*-action}. It is easy to see that $\Lambda_X$ coincides with the conormal variety of $\cS$, i.e. $\Lambda_\cS=\bigcup\limits_{S_w\in \cS}T^*_{S_w}\cB$.  The transversality between the Schubert stratification and the opposite one implies that $\Lambda_X^\infty\cap (\Lambda_X^{opp})^\infty=\emptyset$. The $\bC^*_Z$-action on $T^*\cB$ is the natural $\mathbb{C}^*$-action on the cotangent fibers with weight $1$. It is clear that $k_0=1$ in (\ref{weight assumption}) for this case. 

Let $w_0$ be the longest element in $W$. Let $z_{\alpha}, \alpha\in w_0(\Delta)$ be the linear coordinates around $p_{w_0}$ which correspond to the negative of the simple roots $w_0(\Delta)\subset \fh^*$. Let $F_{w_0}=\sum\limits_{\alpha\in w_0(\Delta)}c_\alpha z_\alpha$ be a generic linear function on $S_{w_0}$, i.e. $\prod\limits_{\alpha\in S}c_\alpha\neq 0$. Then  $L_{w_0, (dF_{w_0})_{p_{w_0}}}$ is the same as  the Lagrangian graph $\Gamma_{dF_{w_0}}$. For any Lagrangian graph, there is a natural brane structure one can put on it, similarly to the case of standard and costandard branes, and this will be  the default brane structure on  $\Gamma_{dF_{w_0}}$. 

For any $w\in W$, let $\fb_w$ denote for the Borel $\Ad_{wB}\fb$, and $\fn_{\fb_w}^-=\bigoplus\limits_{\alpha\in w(\Phi^-)}\fg_\alpha$. Let $N_w^-$ be the unipotent group whose Lie algebra is $\fn^-\cap \fn_{\fb_w}^-$, then each $S^-_w$ is the orbit of $p_w$ under the action of $N_w^-$. The conormal to $S^-_w$ at $\fb_w\in\cB$ is $(\fn^-\cap \fn_{\fb_w}^-)^\perp\cap\fn_{\fb_w}\simeq \fn^-\cap \fn_{\fb_w}= \bigoplus\limits_{\alpha\in w(\Phi^+)\cap \Phi^-}\fg_\alpha$, with respect to the Killing form. Similarly, the conormal at any $\fb_x\in S_w^-$ can be identified with $\fn^-\cap \fn_{\fb_x}\subset  \Ad_{N^-}(\fn^-\cap \fn_{\fb_w})$.

\begin{lemma}[Lemma 5.17 \cite{Nad06}]\label{(co)stalk}
For any sheaf $\cF\in Sh_{\cS}(\cB)$, we have 
$$i_{p_w}^*s_w^*\cF\simeq \mathbb{D}(Hom(\cF, s_{w !}^-\bC_{S_w^-}[\dim S_w^-]))[-\dim S_w^-],$$
$$i_{p_w}^*s_w^!\cF\simeq \mathbb{D}(Hom(\cF, s_{w*}^-\bC_{S_w^-}[\dim S_w^-]))[-\dim S_w^-],$$
for all $w\in W$. 
\end{lemma}

Let $\cN$ and $\cN^{\mathrm{reg}}$ respectively be the nilpotent cone and the orbit of regular nilpotent elements in $\fg$. 
\begin{lemma}\label{no intersect}
For any $w\prec w_0$, $\Gamma_{dF_{w_0}}\cap \overline{\mathfrak{U}}_X(p_w)=\emptyset$.
\end{lemma}
\begin{proof}
Consider the moment maps $\mu_{G}: T^*\cB\rightarrow \cN$ (the Springer resolution) and $\mu_{N}: T^*\cB\rightarrow \fn^*\simeq \fn^-$ of the Hamiltonian $G$-action and $N$-action on $T^*\cB$ respectively, then $\Gamma_{dF_{w_0}}$ is nothing but $\mu_N^{-1}(\bar{e})=N\cdot \mu_G^{-1}(e)$, where $e$ is the image $\mu_{G}(dF_{w_0}|_{p_{w_0}})$ whose projection $\bar{e}$ is the character of $\fn$ corresponding to the linear function $\sum\limits_{\alpha\in w_0(\Delta)}c_\alpha z_\alpha$. It follows from our assumption that $e$ lies in  $\cN^{\mathrm{reg}}$. We only need to show that for any $w\prec w_0$,  $\fn^-\cap \fn_{\fb_w}$ are singular values of $\mu_{G}$, or in other words,   $(\fn^-\cap \fn_{\fb_w})\cap\cN^{\mathrm{reg}}=\emptyset$, because then $\mu_{G}(\overline{\mathfrak{U}}_X(p_w))=\overline{\Ad_{N^-}(\fn^-\cap \fn_{\fb_w})}$ will not intersect $\cN^{\mathrm{reg}}$. 

Note that $\cN^{\mathrm{reg}}\cap \fn^-=\Ad_{B^-}e$. Therefore, if we decompose any element in $\cN^{\mathrm{reg}}\cap \fn^-$ with respect to the weight decomposition, it will have a nonzero component in each negative simple root space. However, the elements in $\fn^-\cap \fn_{\fb_w}$ cannot satisfy this property for $w\prec w_0$, hence we are done. 
\end{proof}

\begin{prop}
The Lagrangian graph $\Gamma_{dF_{w_0}}$ is a Morse brane in $F_{\Lambda_X}(T^*\cB)$.
\end{prop}
\begin{proof}
This directly follows from Theorem \ref{tilting} and Lemma \ref{no intersect}. Alternatively, one can directly use the fact that $\Gamma_{dF_{w_0}}=\mu_N^{-1}(\bar{e})$ to deduce that the Lagrangian is closed. Also from this one easily sees that $\lim\limits_{t\rightarrow 0}a_Z(t)\cdot\Gamma_{dF_{w_0}}\subset\mu_N^{-1}(0)=\Lambda_X$, so $\Gamma_{dF_{w_0}}^\infty\subset \Lambda_X^\infty$. 
\end{proof}

\begin{thm}
The Lagrangian graph $\Gamma_{dF_{w_0}}[\dim_\bC\cB]$ corresponds to the big (indecomposable) tilting perverse sheaf.
\end{thm}
\begin{proof}
We first show that the sheaf corresponding to $\Gamma_{dF_{w_0}}$ plays the role of a Morse kernel on $Sh_{\cS^-}(\cB)$, i.e. calculating vanishing cycles.

To show this, we will use the work \cite{GPS1} and \cite{GPS2} on partially wrapped Fukaya categories. By the proof of Theorem \ref{tilting}, the Lagrangian $\Gamma_{dF_{w_0}}$ can also be viewed as a well defined object in the wrapped Fukaya category $\cW(T^*\cB, \Lambda_{\cS^-}^\infty)$, with stops in the infinity of the conic Lagrangian $\Lambda_{\cS^-}$. We claim that for any object $L\in F_{\Lambda_{\cS^-}}(T^*\cB)$, represented by an isotopy of exact Lagrangians $L_t, t\in [0,1]$, where $L_1=L$ and $L_t, t\in [0,1)$ are all cylindrical, i.e. conic near $\infty$, and $L_{t_1}$ is a negative wrapping of $L_{t_2}$ if $1>t_1>t_2$,  the natural morphism
\begin{align}\label{eq: F to cW functor}
Hom_{F(T^*\cB)}(\Gamma_{dF_{w_0}}, L)\cong \lim\limits_{t\rightarrow 1^-}Hom_{F(T^*\cB)}(\Gamma_{dF_{w_0}}, L_t)\longrightarrow Hom_{\cW(T^*\cB, \Lambda_{S^-}^\infty)}(\Gamma_{dF_{w_0}}, L_0)
\end{align}
is an isomorphism. Since the the morphism (\ref{eq: F to cW functor}) is functorial in $L$, we just need to show the isomorphism for the set of generators $\{L_{S_w^-}, w\in W\}$ consisting of standard branes over the strata. 

Since each $S_{w^-}$ is isomorphic to an affine space, it is easy to see that one can choose an isotopy $L_t, t\in [0,1]$ as above such that $L_1=L_{S_w^-}$, and $L_{t}, t\in [0,1)$ is cofinal in the negative wrapping category $(L_0\rightarrow -)^-$, using \cite[Remark 3.31]{GPS1}. This directly implies that (\ref{eq: F to cW functor}) is an isomorphism for $L=L_{S_w^-}$.

With the claim established, we can use the wrapping exact triangle \cite[Theorem 1.9]{GPS2} and the fact that $\Gamma_{dF_{w_0}}$ intersects the Lagrangian skeleton $\Lambda_{\cS^-}$ transversely at exactly one point on the smooth Lagrangian component $T_{p_{w_0}}^*\cB-\bigcup\limits_{w\prec w_0}\overline{\Lambda}_{S_{w}^-}$, to conclude that $\Gamma_{dF_{w_0}}$ is isomorphic to a linking disc of the Lagrangian component.

 By \cite{GPS3}, we have an equivalence $\cW(T^*\cB, \Lambda_{\cS^-}^\infty)\simeq \Sh^w_{\cS^-}(\cB)$, where $\Sh^w_{\cS^-}(\cB)$ (called \emph{wrapped} microlocal sheaves in \cite{Nad16}) is the full subcategory of compact objects in the large dg-category $\Sh^\diamond_{\cS^-}(\cB)$ of all sheaves constructible with respect to $\cS^-$. Since all strata in $\cS^-$ are contractible, the co-standard sheaves of the strata give the compact generators of $\Sh^\diamond_{\cS^-}(\cB)$, so $\Sh^w_{\cS^-}(\cB)\simeq \Sh_{\cS^-}(\cB)$.  It has been shown in \emph{loc. cit.} that under the equivalence of categories, the linking discs are corresponding to Morse kernels (up to degree shifts). This implies that $\Gamma_{dF_{w_0}}$ represents a Morse kernel on $Sh_{\cS^-}(\cB)$. 
 
 Alternatively, one can use the construction in \cite[Section 4.1, 4.2]{Jin} to represent a Morse kernel by a (partially holomorphic) Lagrangian that intersects $\Lambda_{\cS^-}$ transversely at $(p_{w_0}, e)$ in $\cW(T^*\cB, \Lambda_{\cS^-}^\infty)$. This is isomorphic to a linking disc for the same reason as above, and this gives a proof that $\Gamma_{dF_{w_0}}$ represents a Morse kernel on $Sh_{\cS^-}(\cB)$, without appealing to \cite{GPS3}.

Now by Lemma \ref{(co)stalk}, the stalk and costalk of the sheaf corresponding to $\Gamma_{dF_{w_0}}[\dim_\bC\cB]$ on $S_w$ are concentrated in the right degrees for being a tilting sheaf. It is easy to see this sheaf is exactly the tilting sheaf $T_{p,F}$ for some $p,F$ introduced in \cite{Nad06}.

Lastly, by the multiplicity formula $\mathrm{mult}_{T_w}(T_{p,F})=\dim M_{p,F}(\mathrm{IC}^{opp}_w)$ in \cite{Nad06}, where $T_w$ is the minimal tilting sheaf on $S_w$, 
we see that $\Gamma_{dF_{w_0}}[\dim_\bC\cB]$ corresponds to the big indecomposable tilting sheaf. 

\end{proof}

\end{document}